\documentclass[11pt]{amsproc}
\usepackage{amstext,amsmath,amssymb,amsthm, esint}
\usepackage{latexsym}
\usepackage{exscale}
\usepackage{tikz}

\usepackage{color}

\RequirePackage[cp1251]{inputenc}
\RequirePackage{srcltx}

\numberwithin{equation}{section}
\newcommand\ov{\overline}
\renewcommand\t{\widetilde}

\renewcommand\r{\rangle}
\renewcommand\l{\langle}

\newcommand\dsize{\displaystyle}
\newcommand\supp{\operatorname{supp}}

\newcommand\cal{\mathcal}

\newcommand\R{\mathbb{R}}
\newcommand \C{\mathbb{C}}

\newcommand\Z{\mathbb{Z}}
\newcommand\N{\mathbb{N}}

\newcommand\E{\cal E}

 \newcommand\U{{ \cal U}}
\newcommand\G{{\bf\Gamma}}

\newcommand\B{\mathcal{B}}

\newtheorem{Thm}{Theorem}[section]
\newtheorem{Lemma}[Thm]{Lemma}
\newtheorem{Cor}[Thm]{Corollary}
\newtheorem{Prop}[Thm]{Proposition}
\theoremstyle{remark}

\begin{document}

\title[]{Methods of construction of exponential bases on  planar domains}
\author{Oleg Asipchuk}
\address{Oleg Asipchuk, Florida International University,
	Department of Mathematics,
	Miami, FL 33199, USA}
\email{aasip001@fiu.edu } 
\author{Laura De Carli}
\address{L.~De Carli, Florida International University,
Department of Mathematics,
Miami, FL 33199, USA}
\email{decarlil@fiu.edu}
 \begin{abstract}
  In this paper, we survey and refine several results -- some previously established in the literature -- that facilitate the construction of exponential bases on planar domains with explicit control over the associated frame bounds. We apply our techniques to construct well-conditioned exponential bases on certain planar sets that multi-tile the plane.  
\end{abstract} 
\maketitle
\section{Introduction}

  Let $\Lambda \subset \mathbb{R}^d$ be a discrete set and let $D \subset \mathbb{R}^d$ be a measurable set with positive Lebesgue measure. We consider the  exponential system   $\E=\{e^{2\pi i \lambda_n \cdot x}\}_{n\in\N}, \  x\in D$.
    
  A fundamental and deeply nontrivial question in harmonic analysis is whether the exponential system ${\cal E} $ forms a  Riesz basis for the Hilbert space $L^2(D)$.  We recall that     a sequence of vectors ${\cal V}= \{v_j\}_{j\in\N}$  in a separable Hilbert space  $H$ is a {\it  Riesz basis} if 
  it is complete and there exist  constants $A, \ B>0$    such that   for all finite sequences $ \{a_j\}_{j\in J}\subset\C $ 
  we have that $$
  	A  \sum_{j\in J}   |a_j|^2   \leq  \big\Vert \sum_{j\in J}  a_j  v_j \big\Vert^2  \leq B \sum_{j\in J} |a_j|^2.
$$
    The constants $A, B$ are called the {\it frame bounds}  of the basis. If  ${\cal V}$  is  made of othonormal vectors, then $A=B=1$.
    
 Exponential   bases   have a multitude of applications in mathematics, physics and in other applications.  Then allow to represent solutions  to the Schr\"odinger equation  and other wave equations in terms of  fundamental frequencies and  to reconstruct a functions (for example, a medical image) from its  Fourier  samples.  
  Bases for which  the ratio  of the frame bounds  $\frac BA$ is one or  close to one are called {\it well-conditioned} and are especially relevant in the applications   because  small errors in the samples  (e.g.  from numerical rounding) do not cause large errors in the reconstruction.
 
  Unfortunately there are   few known classes of domains $D$ for which exponential Riesz bases in $L^2(D)$ are known to exist. 
  Some of these notable cases are examined in \cite{AD,CC2018,CL2022,FM2021,KN2015,LS2001,P,Seip1995,Young2000}, among others.

  An important  family of  domains on  which exponential bases are known to exist are those that multi-tile $\mathbb{R}^d$ with a lattice of translations $L$. They are the sets $D$ for which the collection $\{D + \ell : \ell \in L\}$ covers $\mathbb{R}^d$ exactly $k$ times (with the possible exception of a set of measure zero) for some   integer $k\ge 1$   (see, e.g., \cite{GL2014,Kol2015}). 
  
   It is worth mentioning that Y. Lyubarskii and A. Rashkovskii constructed in \cite{LR2000}  an exponential Riesz basis for any convex symmetric polygon whose vertices lie on some lattice. This result was later extended in \cite{DLev2019}, where the existence of such bases was proven for all convex symmetric polygons, without any lattice assumption. These polygons multi-tile $\R^2$; the existence of exponential bases on regular polygons with an odd number of sides is still an open problem.

  \medskip

    While there are very few existence results on exponential bases, there are even fewer non-existence results. 
 G. Kozma, S. Nitzan, and A. Olevskii recently constructed   a disconnected bounded measurable set $S \subset \mathbb{R}$  with the property that no exponential system can form a Riesz basis for $L^2(S)$   \cite{KNO2021}.  
  
  In many of the aforementioned results, the existence of exponential Riesz bases is established using abstract or nonconstructive arguments  and the   frame  bounds of the basis 
  are often not computed explicitly.


 \subsection{Our contribution}
 In this paper, we survey and refine several results that enable the construction of exponential bases on planar domain,  some of them well-conditioned,   with explicit frame bounds. 
 
 Our paper is organized as follows. In Section~2, we present the main tools used throughout the paper and  we complement them with examples which, to the best of our knowledge, are new to the literature.
 
 In Section~3, we apply the methods developed in Section~2, along with results on the eigenvalues of Vandermonde matrices discussed in Section~3.1, to construct a well-conditioned exponential basis on a remarkable non-convex domain.  We hope that our result will serve as a model for similar constructions on other types of domains.
 
 Section~4 is devoted to concluding remarks.

\medskip
\noindent
 {\it Acknowledgments}.  The first author would like to acknowledge the financial support provided by the Florida International University Graduate School Dissertation Fellowship.

\section*{Section 2}
 \setcounter{section}{2}
 In this section, we present results that illustrate how to construct exponential bases on planar domains and how to estimate their frame bounds.
 
 Every bounded planar domain can be approximated by a union of rectangles with sides parallel to the coordinate axes and vertices with rational coordinates. If $D$ is the limit of a sequence  of unions of rectangles $\{D_n\}_{n\in\N}$, then Lemma~\ref{L-limit-bases} shows that a Riesz basis of $L^2(D)$ can be obtained as the union of an increasing sequence of Riesz bases for the spaces $L^2(D_n)$; if  the frame  bounds of these bases are bounded above and below by constants independent of $n$,  then the frame  bounds of the basis of $L^2(D)$ are  also   bounded above and below by the same constants.

 Theorem~\ref{T1-N-interv} shows  how to construct exponential bases  on unions of rectangles with sides parallel to the coordinate axes and vertices with rational coordinates; it also shows that the optimal  frame  bounds  of these bases are  the extremal eigenvalues of certain matrices.
 
 We also present Theorem~\ref{T-Segm-rect}, which  shows how  exponential bases on unions of rectangles, along with their frame  bounds, can be derived from exponential bases on unions of line segments.

 It is also important to understand how subsets of an exponential basis behave under a restriction of the domain. Proposition~\ref{P-compl-Rs} refines results proved in \cite{BCNS2019} and \cite{MM2009}, and provides a hierarchy of implications showing how the frame properties of subsets of a basis $\mathcal{C}$ over a domain $D$ influence the complementary system $\mathcal{C}' = \mathcal{N} \setminus \mathcal{C}$ over the complement domain $D' = \Omega \setminus D$, and vice versa.

 \subsection{Notation and preliminaries}

 We have used   \cite{HJ2013} for standard linear algebra results and the excellent textbook  \cite{Heil}    for definitions and preliminary  material on bases and frames in Hilbert spaces. See also \cite{Chr2016} and   \cite{Young2000}.   
 
For a given  $  v \in\R^d$ and $\rho>0$, we  denote with  $d_\rho:\R^d\to\R^d $ and $\tau_v:\R^d\to\R^d$, $d_\rho(x)=\rho(x)$ and   $\tau_v(x)=v+x$,     the dilation  and translation operators. 

 In the rest of the paper when we refer to a domain   $D \subset \R^d$   we will always assume  that $D$ is  measurable  with  Lebesgue measure $|D| <\infty$. We  denote with $\chi_D :\R^d\to\R$, $\chi_D(x)=\begin{cases} 1 &  if \ x\in  D\cr 0& if \ x\not\in D
 	\cr\end{cases}$  the characteristic function of $D$, and we let 
 $d_\rho D =\{ \rho x \ : \:   x\in D\}$ and  $  \tau_v D =D+  v= \{x+  v \ : \  x\in D\}$.

 We term {\it multi-rectangle}   a finite union of disjoint rectangles with sides parallel to the  disjoint   intervals on the real line.  We denote with $Q_{d}$ the  unit cube $[ -\frac 12,  \,\frac 12)^d$.

 \medskip
 A sequence of vectors ${\cal V}= \{v_j\}_{j\in\N}$  in a separable Hilbert space  $H$ is a {\it frame} if 
 there exist  constants $A, \ B>0$    such that for every $w\in H$,
 \begin{equation}\label{e2-frame}
 	A\|w\|^2\leq  \sum_{j=1}^\infty |\l  w, v_j\r |^2\leq B \|w\|^2.
 \end{equation}
 Here, $\langle\ ,\ \rangle $  and $||\ ||=\sqrt{\l \ , \    \r} $    are the  inner product   and the norm  in $H$.  We say that  ${\cal V}$  is a  {\it Riesz sequence}  if  
 the following
 inequality  is satisfied  for all finite sequences $ \{a_j\}_{j\in J}\subset\C $ 
 \begin{equation}\label{e2- Riesz-sequence}
 	A  \sum_{j\in J}   |a_j|^2   \leq  \big\Vert \sum_{j\in J}  a_j  v_j \big\Vert^2  \leq B \sum_{j\in J} |a_j|^2 ,
 \end{equation}
 and we say that  ${\cal V}$  is a  {\it Bessel sequence} if only the right inequality in \eqref{e2- Riesz-sequence} is satisfied.
 A {\it Riesz basis} is a frame and a Riesz sequence,  and so both    \eqref{e2-frame} and \eqref{e2- Riesz-sequence} are satisfied.
 
 \medskip
  The following   lemma provides a necessary and sufficient condition for a set to form a Bessel sequence.  
 \begin{Lemma}\label{L-Bessel-seq}
 	Let $H$ be a Hilbert space and let  ${\cal V}=\{v_n\}_{n}\subset H$. Then,  ${\cal V}$ is a   Bessel sequence  with bound $B>0$ if and only if  for every set  of constants  $\dsize \{a_n\}\subset\C$ we have
 \begin{equation}\label{e3- Bessel-sequence} 
   \|\sum_n a_nv_n\|^2 \leq B\sum_n|a_n|^2.\end{equation} 
 	
 \end{Lemma}
 \begin{proof}
Assume that  ${\cal V}$ is   a Bessel sequence; letting  $v= \sum_n a_nv_n$, we can write   \begin{align*}
    \|\sum_n a_nv_n\|^2 &= \l  \sum_n a_nv_n,\ v\r = \sum_na_n\l v_n,\, v\r\\
    &\leq  ( \sum_n|a_n|^2)^{\frac 12} (\sum_n |\l v,\, v_n\r|^2)^{\frac 12}\leq \sqrt B\, \|v
 	\| 
\end{align*}
 	from which \eqref{e3- Bessel-sequence}  follows. 
 	
 	Conversely, if \eqref{e3- Bessel-sequence}  holds, the   operator 
 	$T:\ell^2  \to H$, 
 	$$T(\{a_n\}_n)=\sum_n a_nv_n$$ is bounded and its dual is
 	$ T^*: H\to \ell^2 $,
 	$ T^*(v)=\{\l v, v_n\r\}_n.$ 
 	 By general linear algebra results, 
 	$ 
 	\|T^*(v)\|_{\ell^2}= (\sum_n |\l v, v_n\r|^2)^{\frac 12}\leq \sqrt B\|v\|,
 	$ 
 which proves that  ${\cal V}$ is a Bessel sequence.
 \end{proof}
  	
  Let $D\subset  \R^d$.  If $D$ is partitioned into disjoint sets $D_1$, ...,\, $D_m $  which are then translated with translations $\tau_1,\, ..., \tau_m  $ in such a way that the  $D_j+\tau_j$ do not intersect, and let  $\tilde D= (D_1+\tau_1)\cup...\cup (D_m+\tau_m)$. The following Lemma shows  hot wo obtain a basis of  $L^2(\tilde D)$  a basis of    for $L^2(D)$.   Its  proof is  e.g. in \cite{AD}.
 	
 	\begin{Lemma}	\label{l-Transform}
 		 Let $ D=\bigcup_{j=0}^{m} D_j$, where  $|D_j|>0$ for all $j\leq m$ and $k$  $D_k\cap D_j=\emptyset$ when $k\ne j$.  
 		
 		Let $\tau_1.\, ...,\, \tau_m$ be translations  such that 
 		$ ( D_j+\tau_j)\cap   (D_k +\tau_k)  =\emptyset$  when $k\ne j$. Let $\tilde D = \bigcup_{j=1}^m  ( D_j+\tau_j)  $. If 
 		$\B=\{\psi_n(x)\}_{n\in\N} \subset L^2(D)$ is a Riesz basis for $L^2(D)$, then 
 		$$\tilde \B =\Big\{\sum_{j=1}^m \chi_{ \tau_j(D_j)}\psi_n(\tau_j^{-1}x)\Big\}_{n\in\N} $$
 		is a Riesz basis for $L^2(\tilde D)$ with the same frame  bounds. 
        \end{Lemma}

 	The proof of the following easy lemma is omitted.
 	\begin{Lemma}\label{L-dil-basis}
 		Let $  v\in\R^d$ and $\rho>0$. 
 		The set ${\cal E}=\{ e^{2\pi i    \lambda_n \cdot x }\}_{n\in\Z}$ is a Riesz basis  (or a Riesz sequence, or a frame) for $L^2(D)$ with frame  bounds  $A$ and $B$ if and only if the set  $\{ e^{2\pi i   \frac {\lambda_n\cdot x}\rho   }\} $ is a Riesz basis (or a Riesz sequence, or a frame) for  $L^2( d_\rho D  +  v)$ with frame  bounds  $A \rho^{ d}$ and $B \rho^{ d}$.
 	\end{Lemma}

    \subsection{Limit of Riesz bases}
 
 Let $D$ and $\{D_j\}_{j=1}^\infty   $ be measurable subsets of $\R^d$.  We say that  the  sequence   $\{D_j\}_{j=1}^\infty$ is increasing (resp. decreasing) if $D_n\subset D_{n+1}$ (resp. $D_n\supset D_{n+1} $ ) for every  $n\ge 1$.
 
 We say that    $\{D_j\}_{j=1}^\infty$ converges to $D$, and we write $D_n\to D$, if  
 $
 \dsize \lim_{n\to\infty} |D\Delta D_n|\to 0.  
 $
 Here, $ U\Delta V= U\cup V- (U\cap V) $ denotes   the symmetric difference  of    $U$ and $V$. 
  We write $D_n\uparrow D$ (resp. $D_n\downarrow D$) if the sequence $\{D_j\}_{j=0}^\infty$ is increasing,   
  (resp. decreasing) and $D_n$ converges to $D$.

 We show that  a  Riesz basis on $L^2(D)$ can be obtained from an increasing sequence of    Riesz bases on sets that converge to $D$.
 
 \begin{Lemma}\label{L-limit-bases}
 	Let  $\Omega, \,D \subset \R^d$ be  bounded and measurable, with $\Omega \supset D$. Let    $\{D_j\}_{j=1}^\infty $   be  a  family of measurable subsets of $\Omega$, with   $D_n\cap D\uparrow D$ and 
 	$D_n\cup D\downarrow D$.

 	Let   $\B_n \subset L^2(\Omega)\cap L^\infty(\Omega) $ be a Riesz basis of $ L^2(D_n  ) $ with frame  bounds   $A_n$  and $B_n$, and  let $\B=\cup_{n\in\N} \B_n$. Assume   $\ \B_n\subset\B_{n+1}$  for every $n\in\N$ and that 
 	$$ 0<  \alpha\leq A_n\leq B_n\leq \beta <\infty, $$ where   $\alpha,\beta> 0$ are independent of $n$.
 	Assume also that  the $L^\infty$ norm of the functions in  $\B $  is  bounded above by a constant independent of $n$. 
 	Then, $\B $ is a Riesz basis of $L^2(D)$ with frame  bounds  
 	$  \alpha $ and $   \beta.$ 
 \end{Lemma}
 
 \noindent
 Lemma \ref{L-limit-bases}  bears some resemblance with Lemma 1 in \cite{KN2}.  
 \begin{proof}
 	We let  $\B=\{w_n\}_{n\in\Z}$.  
 	Let  $  Y_n= D-  (D\cap D_n) $ and $W_n=(D_n\cup D)-D$.
 	By assumption, $\dsize\lim_{n\to\infty}  | Y_n|=\lim_{n\to\infty}  | W_n|=0$.
 	
 	We show  that $\B$ is a frame in $L^2(D)$ by proving that   
 	\begin{equation}\label{3}  \alpha \| f\|_{L^2(D)}^2 \leq   \sum_{n\in\Z}  |\l f, \, w_n\r_{L^2(D)  }|^2\leq \beta  \| f\|_{L^2(D)}^2  
 	\end{equation}  for every   $f\in L^2(D)$.
 	By Lemma 5.1.7 in \cite{Chr2016} it is enough to prove \eqref{3} for continuous functions with compact support in $D$. 

 	Let $f\in C_0(D)$;  by assumption, $D_n\cap D\uparrow D$ 
 	 and so  $\supp f \subset D\cap D_{\bar N}$ for some $\bar N>0$.  Thus,  $\supp f\subset  D\cap D_N $ for every $N >\bar N$, and   so 
 	$$ 
 	\sum_{m\in\Z}  |\l f, w_m\r_{L^2(D  )  }|^2 = \sum_{m\in\Z}  |\l f, w_m\r _{L^2(D\cap D_N  )}|^2.   $$
 	Since $\B_N$ is a  Riesz basis on $D_N$, it is a frame on $D\cap D_N$ with the same frame bounds.  By assumption,   $\B \supset  \B_N$,  and so for every $N>\ov N$, 
 	
 	\begin{align*}
 	&\sum_{m\in\Z}  |\l f, w_m\r_{L^2(D  )  }|^2 \ge  \sum_{ w_m\in\B_{  N} }  |\l f, w_m\r_{L^2(D  )  }|^2 \\  &= \sum_ {w_m\in\B_{  N} }  |\l f, w_m\r_{L^2(D_{  N}  )  }|^2 \ge A_N\ge \alpha.
 	\end{align*}
 	For the other inequality, we observe that  for every $n\in \N$ we have that 
 	$$
 	\sum_{m\in\Z}  |\l f, w_m\r_{L^2(D \cap D_n )  }|^2 \leq B_n \leq \beta.
 	$$
 	This inequality is valid for every $n\in \N$ and so we also  have $
 	\sum_{m\in\Z}  |\l f, w_m\r_{L^2(D   )  }|^2 \leq \beta.
 	$
 	
 	
 	%

 	\medskip
 	We now prove that   $\B$ is a Riesz sequence in $L^2(D)$. Lemma \ref{L-Bessel-seq} yields that 
 	$
 	\|\sum _{m=1}^{  M} a_m  w_m   \|_{L^2(  D )}^2 \leq \beta $
 	for set of every complex constants $\{a_n\}_{n}  $  with  $\sum _{n=1  }^M|a_n|^2=1$, so we only need to prove the lower bound inequality.

 	Let  $  N>0$  for which   $v_n \in  \B_{  N}$  whenever $n\leq N$.    Recalling that $W_n=D\cup D_n-D$, that  $|W_n|\downarrow 0$  and that $ \|w_n\|_{L^\infty(\Omega)}\leq c$, with    $c$ that does not depend on $n$,  we fix $\epsilon>0$ and  we chose $N$ sufficiently large so that  
 	$ M |W_n|c <\epsilon$ whenever $n>N$.   
 	Since $\B_n$ is a Riesz basis  on $D_n$, it is a Riesz sequence on  $D_n\cup D $, and   
 	$$A_n \leq \|\sum _{m=1}^{  M} a_m  w_m   \|_{L^2(D_n \cup D )}^2 = \|\sum _{m=1}^{  M} a_m w_m   \|_{L^2( D )}^2  + \|\sum _{m=1}^{  M} a_m w_m   \|_{L^2(W_n)}^2.
 	$$
 	Using Cauchy-Schwartz inequality, we obtain
 	\begin{align*}
 		\|\sum _{m=1}^{  M} a_m w_m  \|_{L^2(W_n)}^2 &=\int_{W_n} |\sum _{m=1}^{M} a_mw_m|^2 dx \leq (\sum _{m=1}^{M} |a_n|^2 )
 		\sum _{m=1}^{M}\int_{W_n} | w_n|^2 dx \\
 		&=  |W_n|  \sum _{m=1}^{M}\| w_m\|_\infty^2\leq  M|W_n| c\leq  \epsilon. 
 	\end{align*}
 	We have proved that 
 	$ \dsize 
 	\|\sum _{m=1}^{  M} a_m w_m  \|_{L^2(D)}^2\ge  A_n-\epsilon \ge \alpha -\epsilon.
 	$ 
 	Since this inequality is valid for every $M>0$ and  $\epsilon>0$, we can conclude  that  $\B$ is a Riesz sequence on $L^2(D)$.

 \end{proof}  
 
 \subsection{Exponential bases on multi-rectangles}
  
Let   $Q_d  $ the  unit cube $[ -\frac 12,  \,\frac 12)^d$.  Let $p_1,\, ...,\, p_N\in \R^d $ such that 
$Q_d+  p_k\cap Q_d+  p_h=\emptyset$ when $h\ne k$.   We define the multi-rectangle 
$$ T(  p_1, \, ...,  p_N):=\dsize \bigcup_{j=1}^N (Q_d+  p_k).$$  
When $d=1$, we  use the notation $$\dsize  I(p_1, ...,\, p_N) :=\bigcup_{j=1}^N ([-\frac 12, \frac 12)+  p_k).$$
 
In \cite{D}  the second author of this paper gave an explicit expression for exponential bases and their frame  bounds on sets   $T(  p_1, \, ...,  p_N)$.  Result of the same flavor appeared also in \cite{Kol2015}  and in \cite[Thm. 4.4]{A}.
 \begin{Thm}\label{T1-N-interv}
 With the definition and the notation previously introduced, the set   
 	$$ \B= \B( \delta_1, \, ...,\,   \delta_N):=  \bigcup_{ j=1} ^N \{e^{ 2\pi i \l   n+  \delta_j,\,  x\r}\}_{  n\in\Z^d} $$
 	is a Riesz basis of  $L^2(T(p_1,\, ...,\, p_N))$  if and only if  the matrix
 	\begin{equation}\label{def-Gamma}{\bf \Gamma}=\{ e^{2\pi i \l  \delta_j\cdot  p_k\r}\}_{1\leq j,k\leq N} \end{equation} is nonsingular.  
 	%
 	The optimal frame  bounds of $\B$  are the maximum and minimum singular value of $\G$.
 \end{Thm}
 We recall that the singular values of a square matrix $M$ are the maximum and minimum eigenvalues of   $MM^*$.  
 See also \cite{DK} for a generalized version of this theorem.
 %
 
 %
 

 \subsection{Exponential bases on multi-rectangles  from exponential bases on multi-intervals}
 \medskip
 %

Let  $ T$ be the multi-rectangles
\begin{equation}\label{def-T}T= \bigcup_{k=1}^M [\alpha_k, \, \beta_k)\times [v_k,\, v_k+1),  
\end{equation}
where   $M\in  \N$,   and for every $k\leq M$ we have that $v_k\in\N\cup\{0\}$,  $\alpha_k,\, \beta_k\in\R$, $\alpha_k<\beta_k $ and $v_{k+1}\ge v_k +1$.
 These conditions ensure that the  rectangles $R_k= [\alpha_k, \, \beta_k)\times [v_k,\, v_k+1)$ do not intersect.
 
  Without loss of generality we can assume that $T\subset [0, N]\times [0,  v_M+1]$, where $N\in\N$.   
 We  show how an exponential basis on $L^2(T)$ can be obtained from an exponential basis on a union of segments of the real line. 
  
 \begin{Thm}\label{T-Segm-rect}
 Let $T$ be defined as in \eqref{def-T}; let  $$I= \bigcup_{k=1}^M [\alpha_k + (k-1) N, \  \beta_k +(k-1) N)
 $$ and let 	
 $\B =
  \{e^{2\pi i  \lambda_n x   } \}_{ n\in\Z} $ be an exponential basis on $L^2(I)$ with frame  bounds $A $ and $ B $.  
  Let $\{\omega_n\}_{n\in\Z} \subset \R$ such that
  \begin{equation}\label{e-cond k} N k\lambda_n- \omega_n v_k\in\Z  \quad \mbox{for every $k\leq M$,    $n\in\Z$}.
  	\end{equation}
  Then the set  \begin{equation}\label{e-tB}\t\B =
 	\{e^{2\pi i (\lambda_n x+ (m+\omega_n) y  )} \}_{m, n\in\Z}\end{equation}
 	is an exponential basis on $L^2(T)$ with   frame constant $A$ and $B$.
 	
 	 Conversely, if $ \t\B =
 	 \{e^{2\pi i (\lambda_n x+  (m+\sigma_{ n}) y  )} \}_{m, n\in\Z} $ 
 	 is an exponential basis on $L^2(T)$   with frame  bounds $A $ and $B$ and \eqref{e-cond k} hold, then the set $\B = \{e^{2\pi i \lambda_n x}\}_{n\in\Z}
 	  $  is an exponential basis on $L^2(I)$ with  frame bounds  $A$ and $B$.
 	\end{Thm}
 
  The assumptions on  the multi-rectangle $T$  ensure that  $ I $
 is the union of $M$ disjoint  intervals.
 
   Theorem \ref{T-Segm-rect}  is proved in \cite{DK}, but we sketch the proofs here for the convenience of the reader. Versions of the Lemma below are in \cite{SZ}  and in \cite{DK}.

  \begin{Lemma}\label{L-cart-prod1}
 	Let $D$  be a domain of $\R^d$ and let  $\{\psi_n(x)\}_{n\in\Z}$ be a  Riesz basis of  $L^2(D)$ with frame  bounds $A$ and $B$.
 	Then, for every sequence $\{\omega_n\}_{n\in \N}\subset \R$, the set  
 ${\mathcal U}=\{\psi_n(x) e^{2\pi i(m+\omega_n)y}\}_{n,m \in\Z}$ is Riesz basis of $L^2(D\times [0,1])$ with frame  bounds $A$ and $B$.
 \end{Lemma}
  \begin{proof}[Proof of Theorem \ref{T-Segm-rect}]  If $\B = \{e^{2\pi i \lambda_n x}\}_{n\in\Z}$ is an exponential basis of $L^2(I)$, by Lemma  \ref{L-cart-prod1}  the set  $ 
  	\t\B  
    $ in \eqref{e-tB} is an exponential basis of  $L^2(I\times [0,1))$  with the same frame  bounds.
  We show that $\t \B$ is an exponential basis of $L^2(T)$ as well. 
  %
 Consider  the linear application
  $ L : L^2( I\times [0,1)) \to L^2( T)$, 
  $$ Lf (x, y)=  \sum_{k=1}^M \chi_{R _k} (x) f(x+N(k-1), y-v_k), $$ where    $R_k=[\alpha_k, \, \beta_k)\times [v_k,\, v_k+1) $.
  
  It is not too difficult to verify that $L$ is an invertible isometry, and hence it maps a  Riesz basis of $L^2(I\times [0,1])$  into a  Riesz basis of $ L^2(T)$ with the same frame  bounds. The inverse of $L$ is  the map $ L^{-1} : L^2(T)\to  L^2( I\times [0,1))$  
  $$ L^{-1}g (x, y)=  \sum_{k=1}^M \chi_{I _k} (x) g(x-N(k-1), y+v_k), $$ where   $I_k=[\alpha_k + (k-1)N, \  \beta_k+(k-1)N)\times [0,1) $.
  We have 
  $$L(e^{2\pi i (\lambda_n x+ (m+\omega_n) y  ))})= 
  \sum_{k=1}^M \chi_{R _k}e^{2\pi i [(x+N(k-1))\lambda_n + (y-v_k) (m+\omega_n) ]}
  $$
  $$
  =\sum_{k=1}^M e^{2\pi i (N(k-1)\lambda_n-v_k(m+\omega_n)}\chi_{R _k} (x)e^{2\pi i [x\lambda_n + y(m+\omega_n) ]}
  $$
  and so  the set $L(\t \B)$ is an exponential basis of $L^2(T)$ if the  terms $e^{2\pi i (N(k-1)\lambda_n-v_k(m+\omega_n)}$ that appear in  in the sum do not  depend on $k$. For that we need
   $  Nk\lambda_n-v_k \omega_n  \in\Z $, as required. 
  
  
  %
 
 \end{proof}
 \begin{proof}[Proof of Lemma \ref{L-cart-prod1}]
 	Let us show that $\U$ is a frame. Let $f(x,y)\in L^2(D\times [0,1])$.
 	Fix $n\in\N$;
 	Then, for every $m\in\Z$,
 	\begin{align*}
 		\l f, v_n(x) e^{ 2\pi i(m+\omega_n)y}\r_{L^2(D\times [0,1])}  &=\int_0^1 e^{2\pi i(m+\omega_n)y} \int_D f(x,y)\overline{v_n(x )}dx \\ &=:  \int_0^1 e^{2\pi imy} g_n(y)dy, 
 	\end{align*}
 	where we have let $g_n(y)=: e^{2\pi i \omega_n y} \int_D f(x,y)\overline{v_n(x )}\,dx$.
 	Thus,
 	$$
 	\sum_m|\l   f,\ v_n(x) e^{-2\pi i(m+\omega_n)y}\r_{L^2(D\times [0,1])}|^2= 
 	\sum_m |\l e^{-2\pi imy} g_n(y)\r|^2_{L^2(0,1)}= \|g_n\|_{L^2(0,1) }^2
 	$$
 	\begin{align*}
 	\mbox{and}\quad	&\sum_n\sum_m|\l   f,\ v_n(x)  e^{-2\pi i(m+\omega_n)y}\r_{L^2(D\times [0,1])}|^2  =
 		\sum_n\|g_n\|_{L^2(0,1)}^2\\ &=
 		\sum_n	\int_0^1\left|  \int_D f(x,y)\overline{v_n(x )}dx\right|^2dy\\ &=
 		\int_0^1\sum_n \left|  \l  f(x,\cdot ),\, v_n\r_{L^2(D)}\right|^2dy 
 		\\ &\leq B\int_0^1\left(\int_D|f(x,y)|^2 dx\right)dy= B\|f\|_{L^2(D\times [0,1])}^2.
 	\end{align*}
 	We can prove the lower bound inequality  in a similar manner.
 	
 	\medskip
 	Let us prove that $\B$ is a Riesz sequence.
 	Let $\{a_{n,m}\}$ be a finite set of constants such that $\sum_{n,m}|a_{n,m}|^2=1$. Then,
 	\begin{align*}
 		\|\sum_{n,m} a_{n,m}v_n(x)  e^{-2\pi i(m+\omega_n)y}  \|_{L^2(D\times [0,1])}^2 &=
 		\|\sum_n v_n(x)  \left(\sum_m a_{n,m}  e^{-2\pi i(m+\omega_n)y}\right)\|_{L^2(D\times [0,1])}^2 \\
 		&  =\int_0^1\left(\int_D\|\sum_n v_n(x)b_n(y)\|dx\right)dy,
 	\end{align*}
 	where we have let $b_n(y)=:\sum_m a_{n,m}  e^{-2\pi i(m+\omega_n)y}$.
 	We have
 	\begin{align*}
 		&\|\sum_{n,m} a_{n,m}v_n(x)  e^{-2\pi i(m+\omega_n)y}   \|_{L^2(D\times [0,1])}^2  \leq B\int_0^1\sum_n|b_n(y)|^2 dy
 		\\
 		&	= B \int_0^1\sum_n\left|\sum_m a_{n,m}  e^{-2\pi i(m+\omega_n)y}\right|dy
 		 =B\sum_n\int_0^1\left| \sum_m a_{n,m}  e^{-2\pi i m y}\right|dy
 		\\&=B\sum_n\sum_m|a_{n,m}|^2=B.
 	\end{align*}
 	 The proof of the lower bound inequality is similar. 
 \end{proof}

 \noindent
 {\it Remark}. It is interesting to  observe  that if $\{e^{2\pi i \lambda_n x}\}_{n\in\Z}$ is a Riesz sequence  in $L^2(D)$ with constants $A, B$ and $\{c_n\}_{n\in\Z}\subset \C$  is such that $0< \alpha< |c_n|^2<\beta $ for every $n\in\Z$, then the set $\{c_ne^{\lambda_n x}\}_{n\in\Z}$  is a Riesz sequence  with constants $\alpha A$ and $\beta B$.

    \subsection{  Bases on sub-domains }
 The following proposition  generalizes results by \cite{BCNS2019,MM2009}.
 \begin{Prop}\label{P-compl-Rs}
 	 Let $\Omega \subset \R^d$ and let   ${\cal N}$ be a  Riesz basis for $L^2(\Omega)$ with frame bounds $A,\, B$.  
 	Let $D\subset \Omega$ and ${\cal C} \subset {\cal N}$ ;  
 	Let $D'=\Omega-D$  and ${\cal C}' ={\cal N }-{\cal C }$.
 	Then 1)$\Rightarrow$ 2) $\Rightarrow$3)$\Rightarrow$4)$\Rightarrow$5) $ \Rightarrow$6).
 	
 	\begin{enumerate}
 		\item ${\cal C}$ is a  frame   on $ L^2(D)  $ with bound $\beta <A$ and  $\alpha > 0$

 		\item ${\cal C'}$ is a frame    on $L^2(D) $ with frame bounds $A_1\ge A-\beta$ and $B_1\leq B-\alpha$. 
 		
 		\item  If $B_1<A$, then ${\cal C}'$ is a Riesz sequence    on $L^2(D') $ with frame bounds $A_2\ge A-B_1$   and $B_2\leq  B_1$.

 		\item ${\cal C}'$ is a  Bessel  sequence on $L^2(D')$ with bound $B_3\leq B_2$.
 		
 		\item  If $B_3< A$,  ${\cal C} $ is a frame  on $L^2(D')$ with bound $B_4\leq  B_3$ and $A_4\ge  A- B_3$.

 		\item ${\cal C} $ is a Riesz sequence on $L^2(D )$ with frame bounds $B_5\leq B_4$ and $A_5\ge  A-B_4 $.
 		
 	\end{enumerate}  
 	
 \end{Prop}
 
 \noindent
 {\it Remark.}  This  proposition is valid also when $\alpha=0$, i.e., when ${\cal C}$ in (1) is just a Bessel sequence.     The proof in  \cite{BCNS2019}   only considers the case  $\alpha=0$  and $A=B=1$.
 
 If $\alpha>B-A>0$, then we can estimate the various constants in terms of $\alpha$. Indeed, it is easy to verify that  $A_j\ge  \alpha- (B-A)$ for every $j\ge 2$.
 
 A sufficient condition for (1)---(5) to hold is that $B+\beta <2A$.

 \begin{proof}[Proof of Proposition \ref{P-compl-Rs}]  
 	{\bf  1)$\Rightarrow$ 2)}. Let ${\cal C}=\{w_n\}_{n\in\N}$ and ${\cal C}'=\{v_n\}_{n\in\N}$.  
 	%
 	If ${\cal C}$ is a frame   on $L^2(D)$ with  frame bound  $\alpha$ and $\beta>0$, for every $f\in L^2 (D)$  with compact support in $D$ we  have that
 	$$
 	\sum_m  |\l w_m,\, f\r _{L^2(  D)}|^2  =\sum_m  |\l w_m,\, f\r _{L^2(\Omega)}|^2\ge A\|f\|^2-\sum_n  |\l v_n,\, f\r _{L^2(\Omega)}|^2 $$$$ = A\|f\|^2-\sum_m  |\l v_n,\, f\r _{L^2(D)}|^2\ge (A-\beta)\|f
 	\|^2.
 	$$
 	We can prove that   $ \sum_m  |\l w_m\, f\r _{L^2(  D) }|^2\leq (B-\alpha)\|f\|^2  $ in a similar way, and so we have proved that ${\cal C}'$ is a frame on $D$.
 	
 	\medskip
 	\noindent {\bf  2)$\Rightarrow$ 3)} Assuming that ${\cal C'}$ is a frame  on $D$ with  frame bound  $A_1,\, B_1$,  we prove that  it is a Riesz sequence   on $D'$.
 	
 	Since ${\cal C'}$ is a Bessel sequence on $D$, by Lemma \ref{L-Bessel-seq}  for every sets of constants   $\{b_m\}\subset\C$ such that $\sum_n|b_n|^2=1$, we have that
 	$\|\sum_n b_n w_n\|_{L^2(D)}^2 \leq B_1$. Thus,
 	$$
 	\|\sum_n b_nw_n  \|_{L^2(D' )}^2= \|\sum_n b_nw_n\|_{L^2(\Omega)}^2 -\|\sum_n b_n w_n\|_{L^2(D)}^2\ge A -B_1.
 	$$
 	We also have 
 	
 	$$
 	\sup_{\{b_n\}}\|\sum_n b_nw_n  \|_{L^2(D' )}^2 = B_2\leq B_1
 	$$
 	as required.
 	
 	{\bf  3)$\Rightarrow$4)}   By  Lemma \ref{L-Bessel-seq},   we have that
 	$\sup_{f\in L^2(D')}\sum_m  |\l w_m,\, f\r _{L^2(  D') }|^2=B_3  \|f\|^2  \leq B_2\|f\|^2 $.
 	
 	{\bf  4)$\Rightarrow$5)}   We apply  1) $\Rightarrow$ 2) with $\beta=B_3$. We obtain that ${\cal C} $ is a frame  on $D'$ with bound $B_4\leq  B_3$ and $A_4\ge  A- B_3$.
 	
 	{\bf  5)$\Rightarrow$6)}   If $B_4<A$,  from 2) $\Rightarrow$ 3) follows that ${\cal C} $ is a Riesz sequence on $D$ with frame bounds $B_5\leq B_4$ and $A_5\ge  A-B_4 $.
 \end{proof}
 
 \subsection{Example: An exponential basis for an octagon} We consider the octagon $O$ shown in the figure 1 below. 
 \begin{figure}[ht!]
 	\begin{center}
 		\begin{tikzpicture}\label{  Octagon}
 			
 			\draw [black] (0,0) --(4,0)--(4,4)--(0,4)--(0,0);     
 			\fill[red!40!white,draw=black] (0,0)--(1,0)--(0,1)--cycle;
 			\fill[red!40!white,draw=black] (4,0)--(3,0) --(4,1)--cycle;
 			\fill[red!40!white,draw=black] (0,4)--(0,3) --(1,4)--cycle;
 			\fill[red!40!white,draw=black] (4,4)--(3,4) --(4,3)--cycle;

 			\draw    (-0.1,-0.2) node [black]{{\small $(0,0)$}};
 			\draw    (4.1,4.2) node [black]{{\small $(4,4)$}};
 			\draw    (4.1,-0.2) node [black]{{\small $(4,0)$}};
 			\draw    (-0.1,4.2) node [black]{{\small $(0,4)$}};
 			\draw    (2,2) node [black]{{\small $O$}};
 			\draw    (0.9,-0.2) node [black]{{\small $(1,0)$}};
 			\draw    (2.9,-0.2) node [black]{{\small $(3,0)$}};
 			\draw    (-0.5,0.8) node [black]{{\small $(0,1)$}};
 			\draw    (4.4,0.8) node [black]{{\small $(4,1)$}};
 			\draw    (-0.5,2.8) node [black]{{\small $(0,3)$}};
 			\draw    (4.4,2.8) node [black]{{\small $(4,3)$}};
 			\draw    (0.9,4.2) node [black]{{\small $(1,4)$}};
 			\draw    (2.9,4.2) node [black]{{\small $(3,4)$}};
 			\draw    (-1,2) node [black]{{\small $S$}};

 		\end{tikzpicture}\caption{  Octagon}
 	\end{center}
 \end{figure}

Also, for any   $a,\  b>0$ we  let 
\begin{equation}\label{E-Rhombus}
 	\B(a,b)=\bigcup_{j=0}^1 \left\{e^{2\pi i \left( \left(\frac{n}{a}+\frac{j}{2a}\right)x+   \left(\frac{m}{b}+\frac{j}{2b}\right)y  \right)} \right\}_{n,m\in\Z }.
 \end{equation}
We also let
  \begin{equation*}
 	\B_1=\left\{e^{2\pi i \left( \frac{n}{4}x+   \frac{m}{4}y  \right)} \right\}_{n,m\in\Z}.
 \end{equation*}
 Note that $\B_1$ is the standard exponential basis on the square in Figure 1, and it is easy to verify that $\B(1,1) \subset \B_1$. 
 
We prove the following.
 \begin{Thm}\label{T-octagon}
     The set  $\B_1\setminus \B(1,1)$ is an exponential basis for $L^2(O)$ with frame  bounds $2\leq A$ and $B\leq 14$.
 \end{Thm}

Theorem \ref{T-octagon} will be proven after a few preparatory steps.  The  following celebrated theorem, proved by B. Fuglede in 1974  \cite{Fuglede1974}, shows that any  domain  that tiles $\R^n$ with a lattice of translations possesses an orthogonal exponential basis.
 
Recall that a lattice  is  the image of $\Z^n$ under some invertible linear transformation $L:\R^n\to\R^n.$ The dual lattice $\Lambda^*$ is the set of all vectors $\lambda^*\in\R^n$ such that $\langle\lambda,\, \lambda^*\rangle\in\Z$ for every  $\lambda\in\Lambda$.

 \begin{Thm}[Fuglede's Theorem]\label{T-Fuglede}
 	Let $\Lambda\subset \R^n$ be a lattice.  If $S$ tile  $\R^n$  then the dual lattice $\Lambda^*$ is a spectrum for $S$, and also the converse is true.
 \end{Thm}

 Next, we  use Fuglede's theorem  to construct an exponential basis  of the rhombus $ R_{a,b}$  in Figure 2.   
        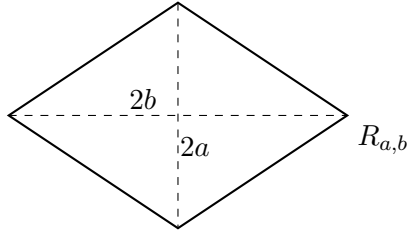
\begin{figure}[ht!]
       \begin{center}
      	\begin{tikzpicture}[scale=1.5]
 		\coordinate (A) at (-1.5, 0);   
 		\coordinate (B) at (0, 1);      
 		\coordinate (C) at (1.5, 0);    
 		\coordinate (D) at (0, -1);     
 		
 		\draw[thick] (A) -- (B) -- (C) -- (D) -- cycle;
 		
 		\draw[dashed] (A) -- (C);
 		\draw[dashed] (B) -- (D);
 		
 	 \node[below right] at (C) {$R_{a,b}$};
 		
 		\node[below] at (.15, -0.1) {$2a$};
 		\node[left] at (-0.1, .15) {$2b$};
 	\end{tikzpicture} \caption{A rhombus with diagonals 2a and 2b}     
       \end{center}

 \end{figure}
 \begin{Cor}
     $\B(a,b)$ defined in \eqref{E-Rhombus} is an orthogonal exponential basis for $L^2(R_{a,b})$ with frame  bounds $A=B=2ab$.
 \end{Cor}

 \begin{proof}
It is easy to see that $R$ tiles $\R^2$ with lattice of translations  $\Lambda=\{(2a n, 2 b m)\cup (2a n+a, 2 b m+b)\}_{m, n\in\Z}$. So, all $\lambda^*\in\Lambda^*$ should satisfy 
$\langle (2a n, 2 b m),\, \lambda^*\rangle  \in \Z $ and $
 		\langle (a, b),\, \lambda^*\rangle   \in \Z
 	 $ whenever  $n,m\in \Z$.  We can see at once that $$\Lambda^*=\Big\{(\frac{n}{a}, \frac{m}{a})\cup (\frac{n}{a}+\frac{1}{2a}, \frac{m}{b}+\frac{1}{2b})\Big\}_{m, n\in\Z}$$  and by Fuglede's theorem, 
 $\B(a,b)$ is an orthogonal Riesz basis for $L^2(R_{a,b})$. The constants $A$ and $B$ are both equal to the measure of $R_{a,b}$, which is $2ab$. 
 \end{proof}

 \begin{proof}[Proof of Theorem \ref{T-octagon}]
    We consider  the octagon $O$ in Figure 2.   We denote the  union of the  four corner triangle in the figure  with $ {\cal T}$.  The triangles can be joined so as to form  the rhombus  $R=R_{1,1}$.   Lemma \ref{l-Transform} yields  that    $\B(1,1)$ is a basis for both $L^2(R_{1,1})$ and $L^2({\cal T})$. Also, $\B_1$ is an orthogonal Riesz basis for  on the rectangle $S=[0, 4]\times [0,4]$ with frame  bounds $A=B=16$. 

  By Proposition \ref{P-compl-Rs} $\B_1\setminus \B(1,1)$ is a Riesz basis for $L^2(O)$ with frame  bounds $2\leq A$ and $B\leq 14$.
 \end{proof}

\section*{Section 3} 
  \setcounter{section}{3}
  In this section  we find a well-conditioned basis on the non-convex domain $D=R_1\cup R_2$ in  Figure 3.  The set $D$ tiles $\R^2$ at level $2$ when translated by the  lattice $\Lambda=\{(n, \frac m4),\ m,\, n\in\Z\} $, i.e., for every $(x,y)\in\R^2$, with the possible exception of a set of measure zero, we have 
 $$
 \sum_{n, m\in\Z} \chi_D (x- n,\, y-\frac m4)=2 .  
 $$
 \begin{figure}[ht] 
 	\begin{center}
 		\begin{tikzpicture}
 			\draw[black, ->] (0, 0)--(5, 0);
 			\draw[black, ->] (0, 0)--(0, 5);
 			\draw [black ] (0,0) --(4,0)--(4,1)--(0,1)--(0,0);
 			\fill[lightgray ](0,0) --(4,0)--(4,1)--(0,1)--(0,0);
 			\draw [black] (0,1) --(0,2)--(2,4)--(4,2)--(4,1)--(2,3)--(0,1);
 			
 			\fill[lightgray ](0,1) --(0,2)--(2,4)--(4,2)--(4,1)--(2,3)--(0,1);
 			\draw    (-0.5,-0.5) node [black]{{\small $(0,0)$}};
 			\draw    (4.5,-0.5) node [black]{{\small $(1, 0)$}};
 			\draw    (2,3.5) node [black]{{\small $R_2$}};
 			\draw    (2,0.5) node [black]{{\small $R_1$}};
 			\draw    (-0.5,1) node [black]{{\small $(0,\frac{1}{4})$}};
 			\draw (-0.5, 2)  node [black]{{\small $(0,\frac{1}{2})$}};
 			\draw    (4.5,1) node [black]{{\small $(1, \frac 14)$}};
 			\draw    (2,4.5) node [black]{{\small $(\frac{1}{2},1)$}};
 			\draw    (2,2.3) node [black]{{\small $(\frac{1}{2},\frac{3}{4})$}};
 			
 		\end{tikzpicture}\caption { The domain $D=R_1\cup R_2$ }
 	\end{center}
  \label{f1}
 \end{figure}
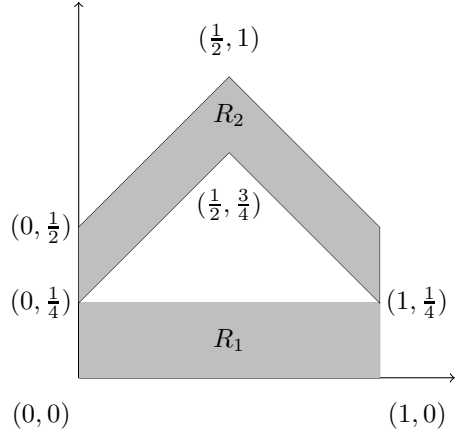

 Every multi-tiling set   has  a Riesz basis of exponentials, and \cite{Kol2015} provides some insight into the structure of this basis.    However, the  results in the aforementioned papers do not allow to construct an exponential basis on $L^2(D)$ and to estimate its frame constant, but we use the methods described in Section 2 and the results stated in the next subsection  to prove the following.

 \begin{Thm}\label{T-main}
 Let $D$ be the domain in Figure 3.	For every $n\in\N$, we let	%
 	$$
 		\B'_N =\bigcup_{j=0}^{2^N-1}\left\{e^{2 \pi i ((2^Nn+\frac{4 j}{2^N})x+(2^Nm+\frac{4 j}{2^N})y)}, \  e^{2 \pi i ((n+\frac{4 j}{2^N}+\frac{2}{2^N})x+(m+\frac{4 j}{2^N}+\frac{2}{2^N})y)}\right\}_{n, m\in\Z}
 $$
 	and we let $\B=\bigcup_{N\in\N}\B'_{N}.$ Then $\B$ 
 	is a Riesz basis for $L^2(D )$, with frame  bounds 
 	$$
 	B=\frac{2+\sqrt{2} }{4} \qquad \text{and } \qquad 
 	A= \frac{2-\sqrt{2} }{4}. 
 	$$
 \end{Thm}

  Before proving the theorem, we need to introduce some results on  the singular values of Vandermonde matrices and other preliminary results. We will so so in the next sub-section.

 \subsection{Exponential bases on multi-rectangles and Vandermonde matrices}\label{Section-EBandMatrices}
  Vandermonde matrices play an important part in the construction of our example.  An  $L\times n$ Vandermonde matrix is a matrix in the form of 
 $$
 A= \left(\begin{matrix}
 	1   &1  & \cdots  &1
 	\cr
 	x_1 & x_2& \cdots & x_n
 	\cr
 	\vdots & \vdots & \vdots & \vdots
 	\cr
 	x_1^{L-1} & x_2^{L-1}  &\cdots & x_n^ {L-1}
 	\end{matrix}\right),
 	$$	
where $x_j\in\C$. The  $x_k $ are sometimes called {\it nodes}. If $n=L$, the determinant of $A$ is given by 
$ \det{A}= \prod_{0\leq  k<h\leq L-1} (x_k-x_h) $, and so $A$ is nonsingular if and only if $x_k\ne x_h$ whenever $k\ne h$.  

We consider Vandermonte matrices  where $x_k= e^{2\pi i y_k}$, with $y_k\in [0,1)$; we refer to $y_k$ as to the 
  {\it frequency} of the node $x_k$. 
 
	Square and rectangular Vandermonde matrices have been extensively studied by numerical and harmonic  
	analysts due to their close relation to polynomial interpolation and approximation,   the problem of mathematical super-resolution and more. See e.g.  \cite{LiLiao2021} and \cite{BATENKOV2021} and the references cited there.

\medskip

Estimating the singular values of Vandermonde matrices is a significant and challenging problem. One approach to this task is the \emph{clustering method}, which involves partitioning the columns of the matrix into groups (clusters).
This technique is discussed in various articles,   \cite{BATENKOV2021, LiLiao2021} among them. It is particularly effective for rectangular Vandermonde matrices, although certain results extend to square cases as well.

Let $ V_k$ 
be a set of column vectors of a Vandermonde matrix $A$;    the angle  between the vectors  of the subspace spanned by  $V_k$ and $V_k'$ is  defined as
\begin{equation}\label{E-angle_def}
\angle_{\min} (V_k,V_{k'}) = \min_{v\in V_k,\, w\in V_{k'}}  \cos^{-1}{\left(\frac{|\langle v,\, w\rangle |}{\|v\|\,\|w\|}\right)}. 
\end{equation}  
 
The following lemma provides a simple tool for clustering the columns of a Vandermonde matrix and for estimating its eigenvalues. 
\begin{Lemma}\label{L-Lemma5.1}
    Let $A\in\C^{N \times s}$,   given in the following block form
    $$
    A=[A_1,...,A_M],
    $$
with $A_{k}\in \C^{N\times s_k}$ and $\sum_{k=1}^M s_k=s$. Let $V_k\subset \C^N$ be the subspace spanned by the columns of the
sub-matrix $A_k$.  Assume that  there exists $\alpha\in [0, \frac 1N]$  such that for all $1\leq k,k'\leq M$, $k\neq k'$ 
\begin{equation}\label{E-angle_condition}
    \angle_{\min} ( V_k,\, V_{k'}) \geq \frac{\pi}{2} - \alpha.
\end{equation}
Let 
$ 
\sigma_1\geq ... \geq \sigma_s
$ 
be the ordered collection of all the singular values of $A$, and let
$$
\tilde \sigma_1\geq ... \geq \tilde \sigma_s 
$$
be the ordered collection of all the singular values of the sub-matrices $\{A_k\}$. Then
\begin{equation}\label{ineq-lemma-3}
\sqrt{1-N\alpha}\, \tilde \sigma_j \leq \sigma_j \leq \sqrt{1+N\alpha}\, \tilde \sigma_j. 
\end{equation}
\end{Lemma}   

For the proof see Lemma 5.1 in \cite{BATENKOV2021}.

\medskip
Lemma \ref{L-Lemma5.1} is quite general and applicable to any matrix. However, the  inner products in the definition \eqref{E-angle_def} involve  sums that are difficult to control. The structure of Vandermonde matrices allows us to express these sum in a more manageable form, as will be demonstrated in the following lemma.

\begin{Lemma}\label{L-sum_of exponents}
    Let $\{a_1,...,a_m\}\subset \R$, $\delta \in \R$, and $L_k=[1,e^{2 \pi i a_k \delta},...,e^{2 \pi i a_k (m-1) \delta}]^T$, where $ T$ denotes the transpose. Assume that  $\delta(a_k-a_{k'})\not\in\Z$  for all $k\neq k'$. Then,
        \begin{equation}\label{E-sin_ratio}
        |\langle L_k,L_{k'}\rangle | = \left\lvert \frac{\sin{(\pi m (a_k-a_{k'}) \delta)}}{\sin{(\pi(a_k-a_{k'}) \delta)}} \right\rvert.
    \end{equation}
    \begin{proof}
       We use the formula for the sum of a geometric series and some trigonometrical simplifications  to write
        \begin{align*}
            |\langle L_k,L_{k'}\rangle | &= |\sum_{j=0}^{m-1} e^{2\pi i (a_k-a_{k'})\delta j} |= \left\lvert \frac{1-e^{2 \pi i m (a_k-a_{k'})\delta}}{1-e^{2 \pi i (a_k-a_{k'})\delta}} \right\rvert \\
            &= \left\lvert \frac{\sin{(\pi m (a_k-a_{k'}) \delta)}}{\sin{(\pi(a_k-a_{k'}) \delta)}} \right\rvert.
        \end{align*}
    \end{proof}
\end{Lemma}

  \subsection{ Proof of Theorem \ref{T-main}}
 
 To   prove Theorem \ref{T-main}, we  use most of the tools described in  sections 2. Our proofs consists of several steps:
\begin{figure}[ht]
	\begin{center}
		\begin{tikzpicture}

			\foreach \y in {0,1,...,3}{
				\foreach \x in {0,1,...,15} {
					\fill[blue!40!white,draw=black] (\x/4,\y/4) rectangle (\x/4+1/4,\y/4+1/4);
			}}

			\foreach \x in {0,1,...,7}{
				\foreach \y in {0,1,...,3} {
					\fill[blue!40!white,draw=black] (\x/4,\y/4+1+1/4+\x/4) rectangle (\x/4+1/4,\y/4+1+1/4+\x/4+1/4);
			}}
			
			\foreach \x in {0,1,...,7}{
				\foreach \y in {0,1,2,...,3} {
					\fill[blue!40!white,draw=black] (\x/4+2,4-1/4-\y/4-\x/4) rectangle (\x/4+2+1/4,4-1/4-\y/4-\x/4+1/4);
			}}

			\draw    (0.8,3.5) node [black]{{\small $R_{N,2} $}};
			\draw    (-0.5,0.5) node [black]{{\small$ R_{N,1} $}};
			
			\draw [red!80!white] (0,0) --(4,0)--(4,1)--(0,1)--(0,0);
			\draw [red!80!white] (0,1) --(0,2)--(2,4)--(4,2)--(4,1)--(2,3)--(0,1);  
		\end{tikzpicture}
	\end{center}
	\caption{Approximation of $D$ by $D_N$}\label{f2}
\end{figure}
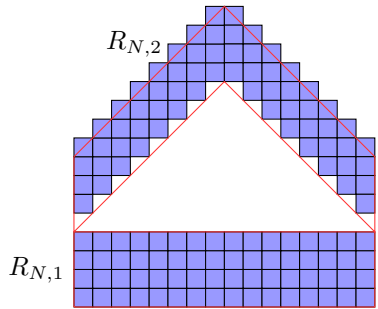

\noindent
{\textbf{a) Approximation}}: We introduce multi-rectangles  $\{D_{ N}\}_{N=2}^\infty = \{R_{N,1} \cup R_{N,2} \}_{N=2}^\infty $  with sides $\frac{1}{2^N}$ (see Figure \ref{f2}), and left bottom vertexes  in the sets $A_{N,1} \cup A_{N,2} $, where 
$ 
A_{N,1}  = \left\{\left(\frac{k}{2^N},\frac{l}{2^N}\right)\right\}_%
{\genfrac{}{}{0pt}{2}{0\leq k \leq 2^N-1}{0\leq l \leq \frac{2^N}{4}-1}}
$ 
and 
$$
 A_{N,2}  = \left\{\left(\frac{k}{2^N},\, \frac{1}{4}+\frac{k+l+1}{2^N}\right),\ 
 \left(\frac{1}{2}+\frac{k}{2^N},\frac{2^N-2-k-l}{2^N}
 \right) \right\}_{\genfrac{}{}{0pt}{2}{0\leq k \leq \frac{2^N}{2}-1}{0\leq l \leq \frac{2^N}{4}-1}}.
 $$
 Letting $D$ be the   domain in Figure 3, $  Y_N= D-  D_N $ and $  W_N= D_N-D$, we can see at once that   $D_N\rightarrow D$, $\dsize\lim_{N\to\infty}  | Y_N|=0$, and $\dsize \lim_{N\to\infty}  | W_N|=0$.

The union of squares in the set $A_{N,1} $   forms the rectangle $R_1=R_{1, N}$ at the bottom of Figure 4.  Moreover, as $N\rightarrow \infty$,  the sequence $R_{1,N}$  convergences to $R_1$ when $N\rightarrow \infty$. Furthermore  $R_{N,1} =R_1=R_{N+1,1} $ for every $N$.
Similarly,  
the  squares  in the sets $A_{N,2} $   form  the figures
$R_{N,2} $ which converges to $R_2$ when $N\rightarrow\infty$. Moreover,  our definitions  ensures that 
$R_{N,2} \subset R_{ N+1, 2}$ for all $N$.

\medskip
\noindent
\textbf{b) Dilation}: Next, we apply  the linear transformation  
$L(x,y)=(2^Nx,2^Ny)$  to obtain unions of squares
 $\{	\tilde{D}_{N}\}_{N=2}^\infty = \{	\tilde{R}_{N,1} \cup 	\tilde {R}_{N,2} \}_{N=2}^\infty $  with sides $1$, and left bottom vertexes in   $\tilde {A}_{N,1} \cup 	\tilde {A}_{N,2} $, where
$  
	\tilde {A}_{N,1}  = \left\{\left(k,l\right)\right\}_{\genfrac{}{}{0pt}{2}{0\leq k \leq 2^N-1}{0\leq l \leq \frac{2^{N}}{4}-1}}
$ 
and 
$$
	\tilde {A}_{N,2} \! =\! \left\{\left(k,\frac{2^N}{4}+k+l+1\right) , \ \left(\frac{2^N}{2}+k,2^N-2-k-l\right)\right\}_{\genfrac{}{}{0pt}{2}{0\leq k \leq \frac{2^N}{2}-1}{0\leq l \leq \frac{2^N}{4}-1}}.
$$

\medskip\noindent
\textbf{ c) Reduction to a one-dimensional problem}: Next, we transform our union of squares into a union of rectangles of height $1$, all aligned on the line $y=0$. To achieve this, we apply the transformation $L$ in Theorem \ref{T-Segm-rect}, with $v_k=k 2^N$ and $w_k=-k$, for $0\leq k\leq N-1$.  

  We represent this union as the Cartesian product of a union of intervals $J_N$ and the interval $[0,1)$. Note that $J_N$ is a union of intervals of length $1$ with left endpoints $L(	\tilde {A}_{N,1} \cup	\tilde {A}_{N,2} )$. More specific, $L(	\tilde {A}_{N,1} )=\{(0,0),(1,0),...,(2^{2N-2}-1,0)\}$ and 
  
  \vskip .7cm
  \noindent
$ 
    L(	\tilde {A}_{N,2} ) = \left\{\left(\frac{2^{2N}}{4}+2^N(k+l+1)+k,0\right), \  \left( 2^{2N}-2^N\left(k+l+\frac{3}{2}\right)+k,0\right)\right\}_{\genfrac{}{}{0pt}{2}{0\leq k \leq \frac{2^N}{2}-1}{0\leq l \leq \frac{2^N}{4}-1}}.
$ 

We claim that for $M=\frac{2^{2N}}{2}$
\begin{equation}
    \B_N= \bigcup_{j=0}^{\frac{M}{2}-1}\left\{e^{2 \pi i (n+\frac{2j}{M})x}, \ e^{2 \pi i (n+\frac{2j}{M}+\frac{1}{2M})x}\right\} _{n\in\Z}
\end{equation}
is a well-conditioned exponential basis for $L^2(J_N)$. To show it, we will use Theorem \ref{T1-N-interv}.

\noindent\textbf{d) Construction of the matrix $\G$}: First, we denote the set of endpoints of $J_N$ as $P_N=P_{N,1}\cup P_{N,2}$, where $P_{N,1}=\{0,1,...,2^{2N-2}-1\}$ and 
 $$
    P_{N,2} = \left\{\frac{2^{2N}}{4}+2^N(k+l+1)+k, \quad  2^{2N}-2^N\left(k+l+\frac{3}{2}\right)+k\right\}_{\genfrac{}{}{0pt}{2}{0\leq k \leq \frac{2^N}{2}-1}{0\leq l \leq \frac{2^N}{4}-1}}.
$$
Also, we denote the set of frequencies by $\Delta_N=\Delta_{N,1}\cup\Delta_{N,2}$, where $\Delta_{N,1}=\{\frac{2j}{M}\}_{0\leq j\leq \frac{M}{2}-1}$ and $\Delta_{N,2}=\{\frac{2j}{M}+\frac{1}{2M}\}_{0\leq j\leq \frac{M}{2}-1}$.  

\medskip
 In view of the applications of  Theorem \ref{T1-N-interv},   we need to evaluate the maximum and minimum singular values of the matrix $\G_N=\{ e^{2\pi i  \delta_j p_k}\}_{1\leq j,k\leq M}$, where $\delta_j$ and  $p_k$ are the elements of $\Delta_N$ and $P_N$, respectively. Note that $|\Delta_N|=|P_N|=M$. 
%
To do so, we  make several observations that will  significantly simplify our calculations. First of all,  
the union of  the intervals with one of the  endpoints in $P_{N,1}$  is   the  interval $[0,M/2)$. Thus, the sets  $\ \bigcup_{j=0}^{\frac{M}{2}-1}\left\{e^{2 \pi i (n+\frac{2j}{M})}\right\}_{n\in\N}$ and   $ \bigcup_{j=0}^{\frac{M}{2}-1}\left\{e^{2 \pi i (n+\frac{2j}{M}+\frac{1}{2M})}\right\}_{n\in\N}$ are orthogonal bases for $L^2([0,M/2))$. Also, by Theorem \ref{T1-N-interv} the matrix   $\G_{N,1}$   (which is a sub-matrix of $\G_N$) for which  $p_k\in P_{N,1}$ and  either $\delta_i\in \Delta_{N,1}$ or $\delta_i\in \Delta_{N,2}$ is orthogonal.

	  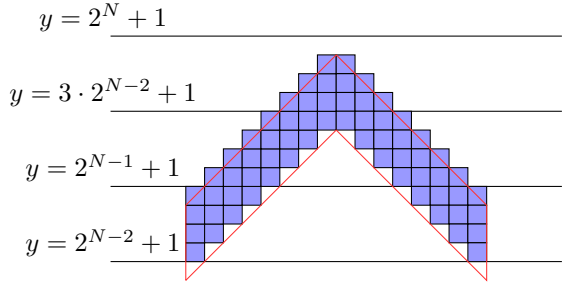
\begin{figure}[ht]
  \begin{center}
     \begin{tikzpicture}
      
      \draw [black] (-1,1.25) --(5,1.25);
      \draw [black] (-1,2.25) --(5,2.25);
      \draw [black] (-1,3.25) --(5,3.25);
      \draw [black] (-1,4.25) --(5,4.25);

   \foreach \x in {0,1,...,7}{
\foreach \y in {0,1,...,3} {
        \fill[blue!40!white,draw=black] (\x/4,\y/4+1+1/4+\x/4) rectangle (\x/4+1/4,\y/4+1+1/4+\x/4+1/4);
   }}

      \foreach \x in {0,1,...,7}{
\foreach \y in {0,1,...,3} {
        \fill[blue!40!white,draw=black] (\x/4+2,4-1/4-\y/4-\x/4) rectangle (\x/4+2+1/4,4-1/4-\y/4-\x/4+1/4);
   }}

 \draw    (-1.1,1.5) node [black]{{\small $y=2^{N-2}+1$}};
 \draw    (-1.1,2.5) node [black]{{\small $y=2^{N-1}+1$}};
 \draw    (-1.1,3.5) node [black]{{\small $y=3\cdot 2^{N-2}+1$}};
 \draw    (-1.1,4.5) node [black]{{\small $y=2^N+1$}};
   \draw [red!80!white] (0,1) --(0,2)--(2,4)--(4,2)--(4,1)--(2,3)--(0,1);
      \end{tikzpicture}
   \end{center}
   \caption{Union of cubes generated by $A_{N,2} $}
 \end{figure}

We claim that also the matrix $\G_{N,2}$  for which  $p_k\in P_{N,2}$ and  either $\delta_i\in \Delta_{N,1}$ or $\delta_i\in \Delta_{N,2}$ is orthogonal.  
To prove this, we return to  the two-dimensional figure. According to our construction,  the set $P_{N,2}$ in one dimension is associated to  $	\tilde {A}_{N,2} $ in two dimensions. All unit squares left bottom vertexes in $	\tilde {A}_{N,2} $ are distributed across four levels, as shown in Figure 5. Each level corresponds to a rectangle with area $\frac{M}{2}$. Moreover, the transformation $ (x,y)\to (x,\  y \mod{2^{N-2}})$  transforms Figure 5  into the    rectangle $A'_{N,2}=[0,2^N)\times[0,2^{N-2})$. Thus, $A'_{N,2} $ is equivalent  to $\{0, 1, ...\,  \frac  M2-1\}$ when transformed into a one dimensional set as we did in the previous step, and the same is true of  $P_{N,2}$.  

 Thus, the  matrix  $\G_{N,2} $ 
 is orthogonal for all $N$.

\medskip
\noindent
\textbf{ e) Estimating   the singular values}: We consider again the matrix $\G_N=\{ e^{2\pi i  \delta_j p_k}\}_{1\leq j,k\leq M}$, with $\delta_j$ and  $p_k$ in $\Delta_N$ and $P_N$, respectively. 
We denote the columns of $\G_N$ with $c_{p_1},\, ..., \, c_{p_{M}}$.

To estimate the singular values of this matrix, we apply Lemma \ref{L-Lemma5.1} and the clustering method described in Subsection \ref{Section-EBandMatrices}. We distribute  the elements of $P_N$ in $\frac{M}{2}$ clusters such that $K_u := \{p\in P_N:p\equiv u-1 \mod \frac{M}{2} \}$ for $1\leq  u\leq \frac{M}{2}$. The fact that $P_{N,1}\equiv P_{N,2} \mod{\frac M2}$  and that $P_{N,1}=\{0, 1, \,... \frac M2-1\}$ ensures that each cluster will contain exactly two points. 

We   show that these clusters are pairwise orthogonal, i.e., the columns of the matrix $\G_N$ associated to the endpoints from two different clusters and frequencies in  $  \Delta_{N}$ are orthogonal. Let $p\neq p' \mod\frac{M}{2}$; 
%
 By Lemma \ref{L-sum_of exponents}, the inner product of the column vectors $c_p$ and $c_{p'}$  is 
\begin{align*}
    |\langle c_p,c_{p'}\rangle | &= \left|\sum_{j=0}^{\frac{M}{2}-1} e^{2\pi i (p-p')\frac{2j}{M}} + e^{2\pi i (p-p')\frac{1}{2M}}\sum_{j=0}^{\frac{M}{2}-1} e^{2\pi i (p-p')\frac{2j}{M}}  \right|   \\
            &=  \left\lvert \frac{1-e^{2 \pi i (p-p')}}{1-e^{2 \pi i (p-p')\frac{2}{M}}} + e^{2\pi i (p-p')\frac{1}{2M}} \frac{1-e^{2 \pi i (p-p')}}{1-e^{2 \pi i (p-p')\frac{2}{M}}}\right\rvert=0,
\end{align*}
because $1-e^{2 \pi i (p-p')\frac{2}{M}}=0$ if and only if $p= p' \mod\frac{M}{2}$.

To apply Lemma \ref{L-Lemma5.1} we need to estimate singular values of all matrices associated to  elements of $P_N$  from the same cluster and frequencies in  $\Delta_{N}$. Let $p,p'\in K_u$, we introduce the   matrix $G_{K_u}=[c_p\,c_{p'}]$. $G_{K_u}$ is a $2\times M$ matrix with   singular values   equal to  the square root of the eigenvalues of the following matrix
$$
S_{K_u}=G_{K_u}^*G_{K_u}= \begin{pmatrix}
\langle c_p,c_p\rangle & \langle c_p,c_{p'}\rangle \\
\langle c_{p'},c_p\rangle & \langle c_{p'},c_{p'}\rangle
\end{pmatrix} = \begin{pmatrix}
M & \langle c_p,c_{p'}\rangle \\
\langle c_p,c_{p'}\rangle & M
\end{pmatrix}.
$$
The characteristic polynomial of $S_{K_u}$  is
$$
\lambda^2 - 2 M \lambda + M^2 - (\langle c_p,c_{p'}\rangle)^2=0.
$$
From  this quadratic equation we obtain, 
\begin{equation*}
    {\sigma'}_{1 }^2=\lambda_{1 }=M + |\langle c_p,c_{p'}\rangle|, \quad {\sigma'}_{ 2}^2=\lambda_{ 2}=M - |\langle c_p,c_{p'}\rangle|. 
\end{equation*}
Now, using the fact that $p= p' \mod\frac{M}{2}$, we calculate
\begin{align*}
    |\langle c_p,c_{p'}\rangle|&= \left|\sum_{j=0}^{\frac{M}{2}-1} e^{2\pi i (p-p')\frac{2j}{M}} + e^{2\pi i (p-p')\frac{1}{2 M}}\sum_{j=0}^{\frac{M}{2}-1} e^{2\pi i (p-p')\frac{2j}{M}}  \right|   \\
    &=\frac{M}{2} |1+e^{2\pi i (p-p')\frac{1}{2M}}| = M \left|\cos{\left(\pi  \frac{(p-p')}{2 M}\right)}\right|= M \left|\cos{\left(\frac{\pi z}{4}\right)}\right|,
\end{align*}
where $z\in\N$ and  $p- p' =\frac{M z}{2}=\frac{2^{2N} z}{4}$. Note that 
$P_N$ is inside the interval $[0,2^{2N-2})$ so $z=1,2,$ or $3$ for $N\geq 2$. Thus, $ |\langle c_p,c_{p'}\rangle|\leq M \cos{\left(\frac{\pi }{4}\right)}=\frac{\sqrt{2} M}{2}$, and
$ 
    {\sigma'}_{1}^2  \leq \frac{2+\sqrt{2} }{2}  M 
\quad \mbox{
and}\quad
{\sigma'}_{2}^2 \geq \frac{2-\sqrt{2} }{2}  M .
$ 

Finally, we use  the inequality \eqref{ineq-lemma-3} with $\alpha=0$ in  Lemma \ref{L-Lemma5.1} to estimate the singular values of the matrix $\G_N$. We obtain:   
$$
    {\sigma}_{1}^2  \leq \frac{2+\sqrt{2} }{2}  M    
 \quad \mbox{
and}\quad
{\sigma}_{N}^2 \geq \frac{2-\sqrt{2} }{2}  M .
$$
Therefore, $\B_N$ is a Riesz basis for $L^2(J_N)$ with frame  bounds
 $$
    B_N  \leq \frac{2+\sqrt{2} }{2}  M, \quad       
A_N \geq \frac{2-\sqrt{2} }{2}  M  .
$$

Lemma \ref{L-cart-prod1} guarantees that  for any $j\leq \frac M2-1$ and any  choice of the   $\{\omega_{n,j}, \ \omega'_{n, j}\}_{n\in\Z}$,  the set 
\begin{align*}
    	\tilde {\B}_N= \bigcup_{j=0}^{\frac{M}{2}-1}\left\{e^{2 \pi i ((n+\frac{2j}{M})x+(m+\omega_{n,j})y)}, \quad e^{2 \pi i ((n+\frac{2j}{M}+\frac{1}{2M})x+(m+\omega'_{n,j})y)}\right\}_{n, m \in\Z} 
\end{align*}
is a Riesz basis for $L^2(J_N\times[0,1))$ with Riesz constants $A_N$ and $B_N$. However, to satisfy condition \eqref{e-tB} in Theorem \ref{T-Segm-rect} we select  $\omega_{n,j}=\frac{2j}{M}$ and $\omega'_{n,j}=\frac{2j}{M}+\frac{1}{2M}$. Thus, by Theorem \ref{T-Segm-rect}, $ 	\tilde {\B}_N$ is a Riesz basis for $L^2(	\tilde {D}_N)$ with Riesz constants $A_N$ and $B_N$. Next, we  use the inverse of the transformation $L$ defined in Step c), which maps $L^2(\tilde {D}_N)$ to $L^2(D_N)$,  and Lemma \ref{L-dil-basis}  to  obtain that
$$
    \B'_N =\bigcup_{j=0}^{\frac{M}{2}-1}\left\{e^{2 \pi i ((2^Nn+\frac{4 j}{2^N})x+(2^Nm+\frac{4 j}{2^N})y)}, \ e^{2 \pi i ((n+\frac{4 j}{2^N}+\frac{2}{2^N})x+(m+\frac{4 j}{2^N}+\frac{2}{2^N})y)}\right\}_{n, m\in\Z} 
$$
is a Riesz basis for $L^2(D_N)$  with frame  bounds 
$$
    B_N'  \leq   \frac{2+\sqrt{2} }{2} \frac{M}{2^{2N}} =    \frac{2+\sqrt{2} }{4} ;  \quad
A'_N \geq \frac{2-\sqrt{2} }{2} \frac{M}{2^{2N}} =    \frac{2-\sqrt{2} }{4}. 
$$
Note that choosing $M=2^N$ guarantees that for $N\geq 2$ we have $\B_{N}'\subset \B_{ {N+1}}'$. Therefore, the conditions of Lemma \ref{L-limit-bases} are satisfied, and
$ 
    \B'=\bigcup_{N=2}^{\infty} \B'_N
$
is an exponential basis for $L^2(D)$ with frame constants $A= \frac{2-\sqrt{2} }{4} $ and $B=\frac{2+ \sqrt{2} }{4} $.

\section *{Section 4}

We  have collected tools  and techniques  that can be used to construct exponential bases with explicit frame constants, and we feel that the examples in sections 2 and 3 can be used as a model for more general constructions.
The domains in our examples  multi-tile the plane, but in principle, our method allows the construction of bases on any set. The challenge is that exponential bases on unions of rectangles approximating these sets may become ill-conditioned, which would  not allow to apply  Lemma \ref{L-limit-bases}.

Theorem \ref{T-main} and Proposition \ref{P-compl-Rs} could potentially be used to construct an exponential basis on a triangle (assuming  that such basis exist).  One way to achieve this could be  to find an exponential  basis  ${ \cal U}$ on the union of the domain  $D$ and the inner triangle  in Figure 3 which contains the basis $\B$  in Theorem \ref{T-main};  if  the frame  constants  of this basis  satisfy the conditions of Proposition \ref{P-compl-Rs},    we can obtain a basis on the inner triangle in figure 3.  
Recall that every   triangles can be obtained from a given triangle via a linear transformation of the plane.
 However, constructing a  basis on  such figure  (and in general, on any planar domain that does not multi-tile) poses   nontrivial challenges and remains an open problem.

\end{document}